\documentclass[a4paper,10pt,leqno]{amsart}

\title{Homotopy coherent diagrams and approximate fibrations}
\author{Wolfgang Steimle}
\address{Universit\"at Bonn\\
               Mathematisches Institut\\
               Endenicher Allee~60,
               D-53115 Bonn, Germany}
\email{steimle@math.uni-bonn.de}
\date{September 2011}


\usepackage[active]{srcltx}
\usepackage{pdfsync}
\usepackage{hyperref}
\usepackage{calc}
\usepackage{enumerate,amssymb, amsmath}
\usepackage{graphicx}
\usepackage[arrow,curve,matrix,tips,2cell]{xy}
  \SelectTips{eu}{10} \UseTips
  \UseAllTwocells

\usepackage{tikz}
\usetikzlibrary{decorations.pathmorphing}

\DeclareMathAlphabet{\matheurm}{U}{eur}{m}{n}

\DeclareMathOperator{\colim}{colim}

\DeclareMathOperator{\cyl}{Hocolim}
\DeclareMathOperator{\Cyl}{cyl}

\DeclareMathOperator{\id}{id}

\DeclareMathOperator{\map}{map}
\DeclareMathOperator{\mor}{mor}

  \newcommand{\IN}{\mathbb{N}}

  \newcommand{\calc}{\mathcal{C}}
  \newcommand{\cald}{\mathcal{D}}
  \newcommand{\cale}{\mathcal{E}}
  \newcommand{\calf}{\mathcal{F}}

  \newcommand{\calu}{\mathcal{U}}

\newcommand{\inv}{^{-1}}

\theoremstyle{plain}
\newtheorem{theorem}{Theorem}[section]

\newtheorem{lemma}[theorem]{Lemma}
\newtheorem{corollary}[theorem]{Corollary}
\newtheorem{proposition}[theorem]{Proposition}

\newtheorem{definition}[theorem]{Definition}
\theoremstyle{definition}

\newtheorem{example}[theorem]{Example}

\newtheorem{remark}[theorem]{Remark}
\newtheorem{fact}[theorem]{Fact}

\theoremstyle{remark}

\newcommand{\arincl}{\ar@{^{(}->}}
\newcommand{\arinclinv}{\ar@{_{(}->}}

\newcommand{\xycomsquare}[8]                
{$$\xymatrix{#1 \ar[r]^-{#2} \ar[d]^{#4} &
    #3 \ar[d]^{#5}  \\
    #6\ar[r]^-{#7} & #8 }$$}

\newcommand{\hc}{{h.c.~}}
\newcommand{\obj}{\operatorname{obj}}
\newcommand{\eval}{\operatorname{eval}}
\newcommand{\op}{^{op}}
\newcommand{\hocolim}{\operatorname{hocolim}}
\newcommand{\Hocolim}{\operatorname{Hocolim}^{\operatorname{BK}}}
\newcommand{\coeq}{\operatorname{coeq}}
\newcommand{\Coh}{\operatorname{Coh}}
\newcommand{\Ho}{\operatorname{Ho}}
\newcommand{\Proj}{\operatorname{Proj}}

\begin{document}

\begin{abstract}
Let $p$ be a fibration over a finite simplicial complex, whose fibers have the homotopy type of finite simplicial complexes. Then $p$ is equivalent to an approximate fibration whose total space is a compact ENR. The proof uses homotopy coherent diagrams and their homotopy colimits. We also comment on the simple homotopy type of the total space and give an application to the fibering of Hilbert cube manifolds.
\end{abstract}

\maketitle

\setcounter{tocdepth}{1}
\tableofcontents

\section{Introduction}

It is well-known that any map $f\colon A\to B$ between topological spaces can be factored as
\begin{equation}\label{intro:fundamental_factorization}
\xymatrix{
A \ar[rr]^\varphi \ar[rd]_f && E \ar[ld]^p\\
& B
}
\end{equation}
where $\varphi$ is a homotopy equivalence and $p$ is a fibration. A standard way to obtain such a factorization involves the space $B^I$ of paths in $B$. While this construction very convenient for many purposes in homotopy theory, the geometric properties of $E$ are very hard to control. For instance, even if $B$ is a finite simplicial complex and the homotopy fibers of $f$ have the homotopy type of finite simplicial complexes, the path-space construction yields a space $E$ which is usually infinite-dimensional.

The purpose of this work is to show that much better control on $E$ can be obtained if we relax the condition on $p$, demanding that it be an \emph{approximate fibration} rather than a fibration. Recall the definition of an approximate fibration:

\begin{definition}
A map $p\colon E\to B$ between topological spaces is an \emph{approximate fibration} if, given any commutative diagram with solid arrows
\[\xymatrix{
X\times \{0\} \ar@{^{(}->}[d] \ar[rr]^{f} && E \ar[d]^p\\
X\times [0,1] \ar[rr]^{h} \ar@{.>}[rru]^H && B  
}\]
and any open cover $\calu$ of $B$, there exists a (dotted) map $H$ such that the upper triangle commutes strictly and the lower triangle commutes ``up to $\calu$'', i.e.~for any $(x,t)\in X\times [0,1]$ there is a $U\in\calu$ such that both $p\circ H(x,t)$ and $h(x,t)$ lie in $U$.
\end{definition}

Approximate fibrations arise naturally in the study of topological manifolds and are closely related to controlled topology. The main result of this work is:

\begin{theorem}\label{intro:main_theorem}
Let $p\colon E\to B$ be a fibration over a finite simplicial complex. Suppose that the fibers of $p$ have the homotopy type of finite simplicial complexes. Then there exists a compact ENR $X$ and a commutative diagram
\[\xymatrix{
X \ar[rr]^\varphi \ar[rd]_q && E \ar[ld]^p\\
 & B
}\]
where $q$ is an approximate fibration and $\varphi$ is a homotopy equivalence.
\end{theorem}

It follows that the fibration $p$ is equivalent in the sense of controlled topology to an approximate fibration whose total space is a compact ENR.

\begin{corollary}\label{intro:main_corollary}
Let $f\colon A\to B$ be a map between topological spaces, where $B$ is a finite simplicial complex and the homotopy fibers of $f$ have the homotopy type of finite simplicial complexes. Then there exists a factorization \eqref{intro:fundamental_factorization} where $\varphi$ is a homotopy equivalence, $p$ is an approximate fibration, and $E$ is a compact ENR.
\end{corollary}

One advantage of $E$ being a compact ENR is that is has a canonical simple homotopy type. So if $A$ in diagram \eqref{intro:fundamental_factorization} is also a compact ENR, we may compute the Whitehead torsion $\tau(\varphi)$. In future joint work with F.T.~Farrell and W.~L\"uck, the author plans to combine this idea with methods of controlled topology to obtain $K$-theoretic obstructions to approximately fibering manifolds. In a similar spirit, here is a prototypical immediate application of Corollary \ref{intro:main_corollary} to infinite-dimensional topology:

\begin{corollary}
Let $M$ be a compact $Q$-manifold, i.e.~a manifold modelled on the Hilbert cube $Q=\Pi_{\mathbb{N}} [0,1]$, and let $f\colon M\to B$ be a map to a compact simplicial complex. If $M$ is simply-connected, then $f$ is homotopic to an approximate fibration.
\end{corollary}

In fact, factoring $f$ as $M\xrightarrow{\varphi} E \xrightarrow{q} B$ as in Corollary \ref{intro:main_corollary} and replacing $E$ by $E\times Q$, we obtain a factorization \eqref{intro:fundamental_factorization} where $E$ is a compact $Q$-manifold. Since $\varphi$ is a simple homotopy equivalence, $\varphi$ is homotopic to a homeomorphism \cite{Chapman(1974)}. 

In a somewhat different direction, Farrell, L\"uck and the author \cite{Farrell-Lueck-Steimle(2010)} have shown that if $p\colon E\to B$ is a fibration over a finite connected simplicial complex with finite fibers, then the total space of $p$ also comes with a preferred simple homotopy type, which is unique for example if the fibers are simply-connected. (In the general case this simple homotopy type depends on several choices, namely of a simple homotopy type of the fiber over some point $b$ and a ``spider'' in $B$, based at $b$.) It is natural to ask how the simple homotopy type of $E$ relates to that of $X$ in Theorem \ref{intro:main_theorem}.   

\begin{theorem}\label{intro:shtype}
If the fibers of $p$ are simply-connected, then $\varphi$ in Theorem \ref{intro:main_theorem} is a simple homotopy equivalence. In the general case, the choice of a simple homotopy type of the fiber over $b$ and the choice of a spider in $B$ at $b$ determine a factorization in Theorem \ref{intro:main_theorem} such that $\varphi$ is a simple homotopy equivalence. 
\end{theorem}

To sketch the proof of Theorem \ref{intro:main_theorem}, assume for simplicity that $B$ is connected. Let $X$ be a finite simplicial complex with the same homotopy type as the fibers of $p$. The idea is to replace, for each simplex $\sigma$ of $B$, the space $E\vert_\sigma:=p\inv(\sigma)$ by $X\times \sigma$ and to glue these together using fiber transport to identify $X\times\sigma$ with, say, $X\times \tau$ over a common face of $\sigma$ and $\tau$. To make sure that the induced projection down to $B$ is indeed an approximate fibration, a deeper analysis of the homotopy coherences between the different fiber transports is needed. A systematic way to do that is to work with homotopy coherent diagrams from the beginning.

More concretely, given the fibration $p$, there is a diagram $F\colon \sigma\mapsto E\vert_\sigma$ indexed on the poset of simplices of $B$, where the maps $E\vert_\sigma\to E\vert_\tau$ for $\sigma<\tau$ are just the inclusions. Replacing $E\vert_\sigma$ by $X$ throughout, we obtain a diagram $G$ which takes values in compact ENRs; however it is no longer strictly commutative but only homotopy coherent. Yet we show that it can be rectified by a strict diagram $G'$ which still takes values in compact ENRs; the homotopy colimit of this diagram then is the space $X$ we are looking for. The fact that the projection down to $B$ is an approximate fibration may be seen as a generalization of Hatcher's iterated mapping cylinder argument \cite{Hatcher(1975)}.

The first part of this paper (sections \ref{section:hc_diagrams} to \ref{section:hc_transformations}) introduces our model for homotopy coherent diagrams and their homotopy colimits. Of course, since we are interested in geometric properties, we need concrete models rather than just models up to homotopy. Following Cordier \cite{Cordier(1982)}, we use the Dwyer-Kan resolution \cite{Dwyer-Kan(1980b)} to define homotopy coherent diagrams while the model for the homotopy colimit is in the spirit of the original definition of Vogt \cite{Vogt(1973)}, with an emphasis on universal properties rather than the ``concrete'' definition. Existence of the universal object is shown using a coequalizer construction. In section \ref{section:relation_to_BK} we show that this model of the homotopy colimit is homotopy equivalent to the Bousfield-Kan-like model \cite{Bousfield-Kan(1972)}. 

Section \ref{section:geometric_properties} starts to investigate geometric properties of the homotopy colimit over finite posets. The main result, Theorem \ref{intro:main_theorem} is proven in section \ref{section:main_result}; Theorem  \ref{intro:shtype} is proven in the following section \ref{section:simple_homotopy_type}.

\subsection*{Acknowledgements. }I am grateful to Bill Dwyer for helpful conversations. This work has been supported by the SFB 878 of the Deutsche Forschungsgemeinschaft and the Hausdorff Center for Mathematics.

\section{Homotopy coherent diagrams}\label{section:hc_diagrams}

The notion of a \hc diagram was introduced by Vogt \cite{Vogt(1973)}. We shall use the following definition, which is due to Cordier \cite{Cordier(1982)}.

Given a small category $\calc$, denote by $\calf\calc$ the \emph{free category} on $\calc$, where $\obj(\calf\calc)=\obj(\calc)$, and a morphism in $\calf\calc$ from $c$ to $d$ is a word of composable non-identity morphisms $(f_n,\dots,f_1)$ (i.e.~the range of $f_i$ is the domain of $f_{i+1}$) such that the domain of $f_1$ is $c$ and the range of $f_n$ is $d$. Composition is given by concatenation of words. The empty word $()$ is allowed if $c=d$ and takes the role of the identity morphism. 

For a morphism $f\in\calc(c,d)$, denote by $\calf f:=(f)\in\calf\calc(c,d)$ the word consisting of the single letter $f$. (If $f=\id$, then $\calf f=()$ is the empty word.) It is not hard to see that $\calf\calc$ is free on the morphisms $\calf f$ in the sense that any collection of maps
\[\mor_\calc(c,d)\to \mor_\cald(F(c),F(d)),\quad (c,d\in\obj(\calc))\]
which sends identities to identities extends uniquely to a functor $F\colon\calf\calc\to\cald$ which takes the prescribed values on the morphisms $\calf f$.

A given functor $F\colon\calc\to\cald$ induces a functor $\calf F\colon \calf\calc\to\calf\cald$ by sending $\calf f$ to $\calf (F(f))$. This makes $\calf$ to a functor from small categories to small categories.

The co-unit functor $U_\calc\colon \calf\calc\to\calc$ is the identity on objects and sends $\calf f$ to $f$. The co-multiplication functor $\Delta_\calc\colon \calf\calc\to \calf^2\calc$ sends $\calf f$ to $\calf^2 f$, thus the word consisting of the single letter $\calf f$. More precisely we obtain natural transformations
\[\calf^2 \xleftarrow{\Delta} \calf \xrightarrow{U} \id\]
which provide $\calf$ with the structure of a co-monad. Thus, letting $\calf_n\calc=\calf^{n+1}\calc$ we obtain a simplicial category (with discrete object set) $\calf_\bullet \calc$. Explicitly, the face and degeneracy functors are the identity on objects, while on morphism sets they are given by
\begin{align*}
d_i=\calf^i(U_{\calf^{n+1-i}\calc})& \colon \calf^{n+1}\calc(c,d)\to \calf^n\calc(c,d)\\
s_i=\calf^i(\Delta_{\calf^{n-i}\calc})& \colon \calf^{n+1}\calc(c,d)\to\calf^{n+2}\calc(c,d).
\end{align*}

The category $\calf_\bullet\calc$ was introduced by Dwyer-Kan \cite{Dwyer-Kan(1980b)}. We will, if necessary, refer to the composition in $\calf_\bullet\calc$ as ``formal composition'' in order to distinguish it from the composition law in $\calc$. The distinction between formal composition and actual composition is indeed the key ingredient in the consideration of homotopy coherence.

\begin{definition}
A \hc diagram of spaces of the shape $\calc$ is a simplicial functor
\[F\colon\calf_\bullet\calc\to\mathbf{Top}.\]
\end{definition}

Here we consider the category $\mathbf{Top}$ of (compactly generated) spaces as enriched over simplicial sets in the usual way. Thus a \hc diagram of the shape $\calc$ is given by
\begin{itemize}
\item For each object $c$ of $\calc$, a space $F(c)$, and
\item for each pair of morphisms, a map
\[\vert\calf_\bullet\calc(c,d)\vert\to\map(F(c), F(d))\]
satisfying the obvious naturality rule. 
\end{itemize}

\begin{example}\label{example:augmentation_functor}
There is a canonical ``augmentation'' functor $\varepsilon\colon\calf_\bullet\calc\to\calc$ (which takes a chain of composable morphisms to its composition). Thus any strict diagram of shape $\calc$, i.e.~any functor $\calc\to\mathbf{Top}$, induces a \hc diagram by precomposition with $varepsilon$.

The augmentation functor induces homotopy equivalences on morphism spaces. In fact, for $c,d\in\obj(\calc)$, the map
\[\calf_n\calc(c,d)\to\calf_{n+1}\calc(c,d)\]
sending a morphism $f$ to the morphism $\calf f$ is a contracting homotopy.
\end{example}

\begin{example}
A functor $f\colon \calc\to\cald$ induces a simplicial functor $\calf_\bullet f\colon \calf_\bullet\calc\to\calf_\bullet\cald$. Hence if we are given a \hc diagram $F$ of shape $\cald$, we may precompose to obtain a \hc diagram $f^* F$ of the shape $\calc$. Notice that if $f$ is full resp.~faithful, then so is $\calf_\bullet f$. In particular, if $f$ is the inclusion of a subcategory, we write $F\vert_\calc$ for $f^* F$.
\end{example}

\begin{example}
It is illuminating to calculate $\calf_\bullet\calc$ for  $\calc=[1]=(0<1)$ and $\calc=[2]=(0<1<2)$.
\begin{enumerate}
\item If $\calc=[1]=(0\xrightarrow{f} 1)$, then 
\[
\calf_0\calc(0,0)=\{\id\},\quad \calf_0\calc(1,1)=\{\id\},\quad
 \calf_0\calc(0,1)=\{\calf f\}
\]
and all the higher simplices are degenerate. Thus $\calf_\bullet\calc=\calc$, and a \hc diagram of shape $[1]$ is just a strict diagram of shape $[1]$, i.e.~a map $F(0)\to F(1)$ of spaces.
\item Now let $\calc=[2]$ as in the diagram
\[\xymatrix{
 & 1 \ar[rd]^g \\
0 \ar[ru]^f \ar[rr]^h && 2
}\]
where $h=g\circ f$. Now $\calf_\bullet\calc(0,2)$ is non-discrete, namely
\[\calf_0\calc(0,2)=\{(h), (g,f)\},\quad \calf_1\calc(0,2)=\{((h)), ((g),(f)), ((g,f))\}
\]
where the last element $((g,f))$ is non-degenerate. Its boundaries are given by the two elements $(g,f)=\calf g\circ\calf f$ and $(h)=\calf(g\circ f)$. All higher simplices are degenerate, so
\[\calf_\bullet\calc(0,2)\cong\Delta^1_\bullet\]
is the 1-simplex. 

Hence a \hc diagram of shape $[2]$ consists of three spaces $F(0)$, $F(1)$, $F(2)$, three maps $F(\calf f)\colon F(0)\to F(1)$, $F(\calf g)\colon F(1)\to F(2)$, and $F(\calf h){\colon}\linebreak[1] {F(0)}\to F(2)$, together with a homotopy
\[F(0)\times I\to F(2)\]
between $F(\calf g)\circ F(\calf f)$ and $F(\calf h)$.
\end{enumerate}
\end{example}

\begin{remark}
For any two objects $c,d$, the simplicial set $\calf_\bullet\calc(c,d)$ comes provided with an ``extra degeneracy'', given by the operator $\calf$. (This operator is not compatible with composition, though.) It satisfies the relations
\[d_0\circ\calf = \id, \quad s_0\circ\calf=\calf^2\]
and
\[d_i\circ\calf = \calf\circ d_{i-1},\quad s_i\circ\calf=\calf\circ s_{i-1}\quad (i>0)\]
where we use the convention that in simplicial degree 0, the map $d_0$ is given by $U$.
\end{remark}

\section{The homotopy colimit}\label{section:hocolim}

Given a small category $\calc$, denote by $\hat\calc$ the category obtained from $\calc$ by adding one terminal object, denoted by $*$. We refer to $\hat\calc$ as the cone of $\calc$.

In the following we are going to consider, for a given \hc diagram $F$ of shape $\calc$, extensions $G$ of $F$ to a \hc diagram of shape $\hat\calc$. 

If $G$ is such an extension, it is easy to see that any map $\varphi\colon G(*)\to Z$ to some space $Z$ induces, by composition, another extension $G'$ of $F$ such that $G'(*)=Z$. Thus we may consider the category where an object is an extension $G$ of $F$ to a \hc diagram of shape $\hat\calc$, and the set of morphisms from $G$ to $G'$ is given by the set of maps $\varphi\colon G(*)\to G'(*)$ such that $G'$ is induced from $G$ by $\varphi$. 

\begin{definition}
A homotopy colimit of $F$ is an initial object in this category.
\end{definition}

In other words, a homotopy colimit of $F$ is an extension $\hat F$ of $F$ to a \hc diagram of shape $\hat\calc$, where we denote $\cyl(F):=\hat F(*)$. It has to satisfy the following universal property: Given any other extension $G$ of $F$, there is a unique map $\cyl(F)\to G(*)$ such that $G$ is induced from $\hat F$ by $f$ in the way described above.

We proceed to show that such a homotopy colimit always exists. Let $\phi_\bullet$ be the contravariant simplicial functor on $\calf_\bullet\calc$, defined by
\[\phi_\bullet(c):=\calf_\bullet\hat\calc(c,*).\]
Then an extension $G$ of $F$ over $\hat\calc$, such that $G(*)=Z$, defines a natural transformation
\[\phi_\bullet(c)\to \map(F(c),Z).\]
Conversely any such natural transformation determines an extension $G$ of $F$ such that $G(*)=Z$.

\begin{proposition}
Let $\calc$ be any small category. Any \hc diagram $F$ of shape $\calc$ has a homotopy colimit.
\end{proposition}

\begin{proof}
We saw above that an extension of $F$ over $\hat\calc$ corresponds to giving a natural transformation
\[\phi_\bullet(c)\to\map(F(c), Z).\]

By adjunction, such a natural transformation is given by a map
\[\coprod_{c\in\obj(\calc)} \vert \phi_\bullet(c)\vert\times F(c) \to Z\]
The naturality property translates to the fact that this map factors through the coequalizer 
\[\coeq\biggl(\coprod_{c,d\in\obj(\calc)} \vert\phi_\bullet(d)\vert\times\vert\calf_\bullet\calc(c,d)\vert\times F(c) \rightrightarrows\coprod_{c\in\obj(\calc)} \vert\phi_\bullet(c)\vert\times F(c)\biggr)\]
of the two obvious maps.

Thus if we denote by $\vert\phi_\bullet\vert \otimes_{\vert\calf_\bullet\calc\vert} F$ this coequalizer, then this space satisfies the universal property of $\cyl(F)$, since an extension of $F$ over $\hat\calc$ corresponds precisely to a map
\[\vert\phi_\bullet\vert \otimes_{\vert\calf_\bullet\calc\vert} F\to Z.\qedhere\]
\end{proof}

Given a functor $f\colon\cald\to\calc$ between small categories and  a \hc diagram $F$ of shape $\calc$, we saw that one may restrict $F$ to a \hc diagram $f^*F$ of shape $\cald$. So there is a canonical map
\[\bar f\colon \cyl(f^* F)\to\cyl(F).\]

In fact, denote by $\hat F$ the ``canonical'' (i.e.~initial) extension of $F$ to a \hc diagram of shape $\hat\calc$, such that $\hat F(*)=\cyl(F)$. Then the restriction of $\hat F$ to $\hat\cald$ is an extension of $F$. By the universal property of $\cyl(f^*F)$ this extension is induced by a unique map which we call $\bar f$.
It is not hard to see that if $g\colon \cale\to\cald$ is another functor, then the two maps
\[\bar f\circ \bar g, \overline{f\circ g}\colon \cyl(g^*f^*F)\to\cyl(F)\]
agree.

We conclude this section by a useful lemma which deals with the case when the underlying category has a terminal object. 

\begin{lemma}\label{lemma1}
Let $F$ be a \hc diagram of shape $\hat\calc$ for some small category $\calc$. Then $\cyl(F)$ is the mapping cylinder of the map
\[\cyl(F\vert_\calc)\to F(*)\]
which is induced by the universal property of the homotopy colimit of $F\vert_\calc$. In particular,
\[\cyl(F)\simeq F(*).\]
\end{lemma}

\begin{example}
If $F$ is a \hc diagram on $\calc=[n]$, then $\cyl(F)$ is an iterated mapping cylinder. These objects have been considered by Hatcher \cite{Hatcher(1975)} in the case where $F$ is a strict diagram. The iterated mapping cylinder comes with a canonical projection to $\Delta^n$; Hatcher showed that if $F$ takes all objects to PL spaces and all morphisms to cell-like mappings, then the projection down to $\Delta^n$ is a PL fibration. In a similar spirit, we will show that if $F$ maps all morphisms to homotopy equivalences, then the projection to $\Delta^n$ is an approximate fibration.
\end{example}

\begin{proof}[Proof of Lemma \ref{lemma1}]
Denote by $\check \calc$ the category $\hat\calc$ with another terminal object added which we denote by $m$. Let $t\colon *\to m$ the unique morphism in $\check\calc$. To prove Lemma \ref{lemma1}, we need to analyze the extensions of $F$ to a \hc diagram of shape $\check\calc$. Therefore we need to analyze $\calf_\bullet\check\calc$.

\emph{Claim. }We have
\begin{enumerate}
\item $\calf_\bullet\check\calc(c,d)\cong\calf_\bullet\hat\calc(c,d)$ if $c,d\neq m$.
\item $\calf_\bullet\check\calc(c,m)\cong\calf_\bullet\hat\calc(c,*)\times\Delta^1_\bullet$ if $c\neq *$. This isomorphism is natural in $c$ and is given by (formal) composition with $t$ when evaluated at $1\in\Delta^1_\bullet$.
\item $\calf_\bullet\check\calc(*,m)=\{\calf t\}.$
\end{enumerate}
Hence an extension $G$ of $F$ to $\check\calc$ amounts to the following data:
\begin{itemize}
\item A space $Z=G(m)$,
\item a map $\varphi\colon F(*)\to Z$ which is given by $G(\calf t)$, and
\item natural maps 
\[\psi_c\colon \calf_\bullet\hat\calc(c,*) \to\map(F(c), \map(I, Z))\]
(adjoint to $\calf_\bullet\check\calc(c,m)\to\map(F(c), Z)$) such that the following diagram commutes:
\[\xymatrix{
{\calf_\bullet\hat\calc(c,*)} \ar[rr]^F \ar[d]^{\psi_c}
 && {\map(F(c), F(*))} \ar[d]^{t} \\
{\map(F(c), \map(I, Z))} \ar[rr]^{\eval_1} && {\map(F(c), Z)}
}\]
\end{itemize}
Now notice that, since the maps $\psi_c$ are natural, they just define an extension $H$ of $F\vert_\calc$ over $\hat\calc$, such that $H(*)=\map(I,Z)$. Thus, using the universal property, we obtain a unique map
\[\cyl(F\vert_\calc)\to \map(I,Z)\]
which is adjoint to a map
\[\alpha\colon \cyl(F\vert_\calc)\times I\to Z.\]
Since the restriction along $1\in I$ factors through $F(*)$, the maps $\alpha$ and $\varphi$ define a unique map on the mapping cylinder of
\[\cyl(F\vert_\calc)\to F(*)\]
as claimed.

We still need to prove the claim. Part (i) holds as $\calf_\bullet\hat\calc$ is a full subcategory of $\calf_\bullet\check\calc$. Part (iii) follows from the fact that $t$ allows no non-trivial factorization in $\check\calc$. The only non-trivial part is thus (ii). We are going to write down an explicit isomorphism
\[\iota_\bullet\colon \calf_\bullet\hat\calc(c,*)\times\Delta^1_\bullet \to \calf_\bullet\check\calc(c,m).\]
On 0-simplices, the map $\iota_0$ is given by the following rule: 
\[\bigl((f_k,\dots,f_1), \alpha\colon [0]\to[1]\bigr)\mapsto 
 \begin{cases}
  (t,f_k,\dots, f_1), & \alpha=1,\\
  (t\circ f_k, f_{k-1},\dots, f_1), & \alpha=0.
 \end{cases}
\]
It is easy to see that this map is indeed a bijection, as $*$ is terminal in $\hat\calc$ and $m$ is terminal in $\check\calc$.

The map $\iota_n$ is described by the following inductive rule:
\[\bigl((f_k, \dots, f_1), \alpha\bigr)\mapsto
 \begin{cases}
  (\calf^n t,f_k,\dots, f_1), & \alpha\equiv 1,\\
  (\iota_{n-1}(f_k, d_0\alpha), f_{k-1}, \dots, f_1),& \mathrm{otherwise}.
 \end{cases}
\]
We show by induction that this is a bijection. Indeed it is clear that $\iota_n$ is a bijection between $\calf_n\hat\calc(c,*)\times \{1\}$ and those words in $\calf_n\check\calc(c,m)$ whose first letter is $\calf^n t$. On the other hand, $d_0$ is a bijection
\[d_0\colon \Delta^1_n-\{1\}\to \Delta^1_{n-1}\]
so the rule $(f_k,\alpha) \mapsto \iota_{n-1}(f_k, d_0\alpha)$ is a bijection 
\[\calf_{n-1}\hat\calc(\operatorname{domain}(f_k),*)\times \Delta^1_n-\{0\}\to \calf_{n-1}\check\calc(\operatorname{domain}(f_k),m)\]
by the inductive assumption.
\end{proof}

\begin{remark}
Here is an informal description of the map $\iota_n$: An element $x$ of $\calf_n\hat\calc(c,*)$ is a word of words of ... of morphisms of $\hat\calc$. Sticking to our convention that any word is surrounded by parentheses, such an element has thus the form
\[((\dots (a_1,a_2,\dots),(b_1,b_2,\dots),\dots )\]
where we have $(n+1)$ opening parentheses at the beginning. Now let $\alpha\colon [n]\to [1]$ be an $n$-simplex of $\Delta^1$. If $\alpha$ is constantly 0, then $\iota_n$ composes the very first morphism of $\hat\calc$ appearing in the word $x$ with the morphism $t$. Otherwise let $i$ be the first index such that $\alpha(i)=1$. Then $\iota_n$ introduces into $x$, after the $(i+1)$st opening parenthesis, the word given by $t$.
\end{remark}

\section{Relation to the homotopy colimit in the sense of Bousfield-Kan}\label{section:relation_to_BK}

If $F\colon \calc\to\mathbf{Top}$ is a strict diagram, then Bousfield-Kan \cite{Bousfield-Kan(1972)} defined its homotopy colimit to be
\[\hocolim(F):=\vert {-}\backslash\calc \vert\otimes_\calc F.\]
Here ${-}\backslash\calc$ is the category-valued functor on $\calc\op$ which sends $c$ to the under category $c\backslash\calc$; $\vert\cdot\vert$ denotes the geometric realization.

This definition readily generalizes to \hc diagrams: Let
\[\psi_n(c):=\coprod_{c_0,\dots,c_n\in\obj(\calc)} \calf_n\calc(c_{n-1},c_n)\times \dots\times \calf_n\calc(c_0,c_1)\times\calf_n\calc(c,c_0).\]
Then $\psi_\bullet(c)$ is in fact the diagonal of a bisimplicial set where the horizontal simplicial operations are the ones of $\calf_\bullet\calc$. The vertical differential $d_i$, for $i<n$ is given by the composition rule
\[\calf_\bullet\calc(c_i,c_{i+1})\times\calf_\bullet\calc(c_{i-1}, c_i)\to\calf_\bullet\calc(c_{i-1}, c_{i+1}).\]
(Here we let $c_{-1}:=c$.) The $n$-th vertical differential just lets the first factor out, and the degeneracies introduce identity morphisms.

Then $\vert\psi_\bullet\vert\colon\calf_\bullet\calc\op\to\mathbf{Top}$, so one may define
\[\Hocolim(F):=\vert\psi_\bullet\vert\otimes_{\vert\calf_\bullet\calc\vert} F.\]

The contraction $\varepsilon$ induces a natural transformation
\[\varepsilon_\bullet \colon \psi_\bullet(c)\to N_\bullet(c\backslash\calc)\]
so that if $F$ is a strict diagram, we obtain a natural homotopy equivalence
\[\varepsilon\colon \Hocolim(\varepsilon^*F)\to\hocolim(F).\]

\begin{theorem}\label{thm:relation_to_hocolim}
\begin{enumerate}
 \item For a \hc diagram $F$, there is a homotopy equivalence
\[\kappa\colon \Hocolim(F)\xrightarrow{\simeq}\cyl(F).\]
\item If $F$ is a strict diagram, then there is a homotopy equivalence 
\[\tau\colon \cyl(\varepsilon^*F)\xrightarrow{\simeq}\hocolim(F)\]
such that the following diagram commutes:
\[\xymatrix{
{\Hocolim(\varepsilon^*F)} \ar[rr]^\kappa_\simeq \ar[rd]_\varepsilon^\simeq && {\cyl(\varepsilon^*F)} \ar[ld]^\tau_\simeq \\
 & {\hocolim(F)} 
}\]
\end{enumerate}
\end{theorem}

The remainder of this section is devoted to the proof of this Theorem. The natural transformation in (i) is induced by simplicial a map
\[\kappa_\bullet \colon \psi_\bullet(c)\to\phi_\bullet(c)\]
which is natural in $c$. 

On 0-simplices, $\kappa_0(f_0):=\calf t\circ f_0$. Here $f_0\in\calf_0\calc(c,c_0)$, and $t\colon c_0\to*$ is the unique morphism in $\calc$. The composition $\calf t\circ f_0$ is, of course, the (formal) composition in $\calf_0\calc$, thus given by concatenation of words.

For $n\geq 1$, we define inductively
\[\kappa_n(f_n,\dots f_0):=\calf\kappa_{n-1}(d_0 f_n,\dots, d_0 f_1)\circ f_0.\]
This is certainly natural in $c$, i.e.~
\[\kappa_n(f_n,\dots, f_0\circ g)=\kappa_n(f_n,\dots, f_0)\circ g.\]

We have to check that the simplicial identities are satisfied. For example, remembering that the simplicial operations in $\psi_{\bullet}$ are the diagonal ones. Moreover, one can check by hand that
\begin{align*}
d_0\kappa_n(f_n,\dots,f_0) & = \kappa_{n-1}(d_0f_n,\dots d_0 f_1)\circ d_0 f_0\\
 & =\kappa_{n-1}(d_0f_n,\dots d_0f_1, d_0(f_1\circ f_0))\\
 & =\kappa_{n-1}\circ d_0(f_n,\dots f_0),
\end{align*}
\[d_1\circ \kappa_1= \kappa_0\circ d_1\]
and by induction on $n$ that
\begin{align*}
d_i\circ\kappa_n(f_n,\dots f_0) & = d_i\calf\kappa_{n-1}(d_0f_n,\dots d_0f_1)\circ d_i f_0\\
& = \calf d_{i-1}\kappa_{n-1}(d_0f_n,\dots, d_0 f_1)\circ d_i f_0\\
 & =\calf\kappa_{n-2} d_{i-1}(d_0 f_n,\dots, d_0f_1)\circ d_i f_0\\
 & = \calf\kappa_{n-2}(d_0 d_i f_n,\dots,d_0 d_i f_{i+1}\circ d_0 d_i f_i,\dots, d_0 d_i f_1)\circ d_i f_0\\ 
 & =\kappa_{n-1}\circ d_i(f_n,\dots, f_0)
\end{align*}
for $i>0$. The calculation for the degeneracy operators is similar.

Theorem \ref{thm:relation_to_hocolim} follows from the Lemma down below after geometric realization.

\begin{lemma}
The map $\kappa_\bullet$ is a homotopy equivalence of simplicial functors $\calf_\bullet\calc\to\mathbf{sSet}$. More precisely, there is a simplicial natural tranformation
\[\lambda_\bullet\colon \phi_\bullet(c)\to\psi_\bullet(c)\]
such that $\kappa_\bullet\circ\lambda_\bullet=\id$, and a simplicial homotopy
\[H_\bullet\colon\psi_\bullet(c)\times \Delta^1_\bullet\to\psi_\bullet(c)\]
from $\lambda_\bullet\circ\kappa_\bullet$ to the identity, which is also natural in $c$.
\end{lemma}

\begin{proof}
To define a natural map 
\[\lambda_n\colon\calf_n\hat\calc(c,*)\to \coprod_{c_0,\dots,c_n} \calf_n\calc(c_{n-1},c_n)\times\dots\times\calf_n\calc(c,c_0),\]
it suffices to define it on words of the length one, i.e.~on elements of the form $\calf f$.

The maps $\lambda_n$ are defined inductively by
\[\lambda_0(\calf f)=\id, \quad\lambda_n(\calf f)=(\calf \lambda_{n-1}(f), \id).\]
here the last expression denotes the $(n+1)$-tuple whose last component is the identity map, and the first $n$ components are obtained from $\lambda_{n-1}(f)$ by applying $\calf$ component-wise.

Now $\kappa_0\circ\lambda_0$ is obviously the identity. Inductively one concludes
\[\kappa_n\circ\lambda_n(\calf f)  = \kappa_n(\calf\lambda_{n-1}(f), \id) = \calf\kappa_{n-1}(\lambda_{n-1}(f))\circ\id = \calf f,\]
hence $\kappa_n\circ\lambda_n=\id$. Another inductive argument shows that
\[\lambda_n\circ\kappa_n(f_n,\dots, f_0)=(x_n f_n,\dots,x_0 f_0)\]
where we abbreviate
\[x_j:=\calf^j d_0^j.\]

To show that $\lambda_\bullet$ is simplicial, one could argue inductively again as in the case of $\kappa_\bullet$. But as $\kappa_n\circ\lambda_n=\id$, the simplicial identities hold for $\lambda_\bullet$ provided they hold for $\kappa_\bullet$ and $\lambda_\bullet\circ\kappa_\bullet$; so they are a consequence of the fact that the map $H_\bullet$ defined below is simplicial.

In fact, $H_\bullet$ will be a simplicial homotopy between $\lambda_\bullet\circ\kappa_\bullet$ and the identity on $\psi_\bullet(c)$. It is defined by the formula
\begin{align*}
 H_n\colon \psi_n(c)\times\Delta^1_n & \to \psi_n(c)\\
(f_n,\dots,f_0,\alpha)& \mapsto (x_{n\wedge k(\alpha)} f_n, \dots, x_{0\wedge k(\alpha)}f_0)
\end{align*}
where the sign $\wedge$ denotes the infimum and
\[k(\alpha)= \min\{i;\;\alpha(i)=1\}\in\{0,\dots,n+1\}\quad(\alpha\colon [n]\to[1]).\]

The fact that $H_\bullet$ is simplicial follows from the formulae 
\[k(d_i\alpha)=
 \begin{cases}
   k(\alpha),&i\geq k(\alpha),\\
   k(\alpha)-1, & i<k(\alpha)
 \end{cases},
\quad
k(s_i\alpha)=
 \begin{cases}
   k(\alpha),&i\geq k(\alpha),\\
   k(\alpha)+1, & i<k(\alpha)
 \end{cases},
\]
and
\[
d_i x_j =
\begin{cases}
 x_{j-1}d_i, & i<j,\\
 x_j d_i, & i\geq j
\end{cases},
\quad 
s_i x_j =
\begin{cases}
 x_{j+1}s_i, & i<j,\\
 x_j s_i, & i\geq j
\end{cases},
\]
For example,
\[H_{n-1}\circ d_i(f_n,\dots,f_0,\alpha)=d_i\circ H_n(f_n\dots, f_0,\alpha)\]
as both are given by $(g_{n-1},\dots, g_0)$ where
\[g_j=
 \begin{cases}
  x_{j\wedge k(\alpha)} d_i f_j, & j<i,\\
  x_{j\wedge k(\alpha)} d_i(f_{j+1}\circ f_j), & j=i,\\
  x_{k(\alpha)} d_i f_{j+1}, & j>i\geq k(\alpha),\\
  x_{j\wedge(k(\alpha)-1)} d_i f_{j+1}, & j>i, k(\alpha)>i.
 \end{cases}
\]
The case of degeneracies is similar.
\end{proof}

Next we construct a natural transformation
\[\tau_\bullet\colon\phi_\bullet(c)\to N_\bullet(c\backslash\calc)\]
such that the following triangle commutes:
\[\xymatrix{
{\psi_{\bullet}(c)} \ar[rr]^{\kappa_\bullet} \ar[rd]_{\varepsilon_\bullet} && {\phi_\bullet(c)} \ar[ld]^{\tau_\bullet}\\
& {N_\bullet(c\backslash\calc)}
}\]

To do this, to each morphism $(f_n,\dots,f_1)$ in $\hat\calf_n\calc(c,*)$ we need to give a string
\[c\xrightarrow{\alpha_0}c_0\xrightarrow{\alpha_1}\dots \xrightarrow{\alpha_n} c_n\]
in $\calc$. The rule is defined as follows: Let $\alpha_0:=\varepsilon(f_{n-1},\dots f_1)$, which is  a morphism from $c$ to the domain of $f_n$. Now, by definition, $f_n=(g_m,\dots,g_1)$ is itself a word of morphisms in $\calf_{n-2}\hat\calc$, and we let $\alpha_1:=\varepsilon(g_{m-1},\dots,g_1)$. Iterating this procedure, one obtains maps $\alpha_2,\dots\alpha_n$ as claimed. It is not hard to see that the diagram commutes.

\section{Homotopy coherent transformations and rectifications}\label{section:hc_transformations}

The next task will be to show that a ``\hc transformation'' between \hc diagrams induces, up to homotopy, a unique map between homotopy colimits.

Denote, for $i\leq n$, by $\calc_i\subset\calc\times [n]$ the full subcategory of objects $(c,i)$, which of course is isomorphic to $\calc$. Denote moreover, for $i<j$, by $\calc_{i\dots j}\subset\calc\times [n]$ the full subcategory of all objects of $\calc_k$, where $i\leq k\leq j$. 

\begin{definition}[Vogt, Cordier]
\begin{enumerate}
\item Let $F,G$ be \hc diagrams of shape $\calc$. A \hc transformation from $F$ to $G$ 
\[H\colon F\Rightarrow G\]
is a \hc diagram of shape $\calc\times [1]$ which restricts to $F$ over $\calc_0$ and to $G$ over $\calc_1$.
\item Let 
\[\xymatrix{
 & F_1 \ar@{=>}^{H_0}[rd] \\
F_0 \ar@{=>}^{H_2}[ru] \ar@{=>}^{H_1}[rr] && F_2
}\]
be \hc transformations. We write
\[H_1\simeq H_2\circ H_0\]
if there is a \hc diagram $H$ of shape $\calc\times [2]$ such that $\partial_i H= H_i$ for $i=0,1,2$.
\item Two \hc transformations $H,H'\colon F\Rightarrow G$ are homotopic if $H'\simeq H\circ\id_G$. Here $\id_G$ denotes the restriction of $G$ along the projection $\calc\times [1]\to\calc$.
\item The category $\Coh(\calc,\mathbf{Top})$ has as objects the \hc diagrams of shape $\calc$, and homotopy classes of \hc transformations as morphisms. The composite $[H_2]\circ[H_0]$ is represented by any $H_1$ such that $H_1\simeq H_2\circ H_0$.
\end{enumerate}
\end{definition}

\begin{remark}
One can show that homotopy between \hc transformations is indeed an equivalence relation, and that composition in $\Coh(\calc,\mathbf{Top})$ is well-defined. This relies the fact that the homotopy coherent nerve of the category $\mathbf{Top}$ is a weakly Kan.
\end{remark}

Let $H\colon F\Rightarrow G$ be a \hc transformation. By naturality with respect to the index category, there are two inclusions
\[\cyl(F) \xrightarrow{i_1} \cyl(H) \xleftarrow{i_2} \cyl(G).\]
The map $i_2$ is a homotopy equivalence by the following Lemma:

\begin{lemma}\label{lemma3}
Let $F$ be a \hc diagram of the shape $\calc\times [n]$. Then for any $i\leq n$ the inclusion
\[\cyl(F\vert_{\calc_{i\dots n}})\to \cyl(F)\]
is a homotopy equivalence.
\end{lemma}

\begin{proposition}
The homotopy colimit defines a functor
\[\cyl\colon \Coh(\calc,\mathbf{Top})\to \Ho(\mathbf{Top})\]
by the following rule: If $H\colon F\Rightarrow G$ is a \hc transformation, then 
\[H_*\colon\cyl(F)\to \cyl(G)\]
is the composite $i_2\inv\circ i_1$.
\end{proposition}

\begin{proof}
We need to show that $(K\circ H)_*\simeq K_*\circ H_*$. Therefore let $L$ be a \hc diagram of shape $\calc\times [2]$ restricting to $H$, $K$, and $K\circ H$ respectively. Let $F_i:=H\vert_{\calc_i}$ for $i=0,1,2$. Naturality follows from the commutativity of the following diagram:
\[\xymatrix@C=2.5ex{
 && {\cyl(F_1)} \ar[ld]_\simeq \ar[d] \ar[rd] \\
& {\cyl(H)} \ar[r] & {\cyl(L)} & {\cyl(K)} \ar[l]_\simeq \\
{\cyl(F_0)} \ar[rr] \ar[ru] \ar[rru] && {\cyl(K\circ H)} \ar[u] && {\cyl(F_2)} \ar[ll]_\simeq \ar[llu]_\simeq \ar[lu]_\simeq 
}\]
\end{proof}

\begin{remark}\label{rem:commutativity_of_triangle}
Hence, a diagram
\[\xymatrix{
{\cyl(F)} \ar[rr]^{H_*} \ar[rd]_f && {\cyl(G)} \ar[ld]^g\\
& Z
}\]
commutes up to homotopy whenever there is a \hc diagram of shape $\widehat{\calc\times [1]}$ which extends $H$ and both extensions of $F$ and $G$ as given by the maps $f$ and $g$.
\end{remark}

\begin{remark}
In general it does not make sense to ask for a functorial way to assign a map $H_*$ (in $\mathbf{Top}$) to a \hc transformation, since composition of to \hc transformations is well-defined only up to homotopy. However if $H\colon F\to G$ is a natural transformation between \hc diagrams, i.e.~between functors on $\calf_\bullet\calc$, then there is an induced map $H_*$ in $\mathbf{Top}$, and this assignment is natural.

In fact, given the canonical extension $\hat G$ of $G$ over the cone $\hat\calc$, the maps
\[\calf_\bullet\hat\calc(c,*) \xrightarrow{\hat G} \map(G(c),\cyl(G))\xrightarrow{H_c^*} \map(F(c),\cyl(G))\]
assemble to an extension of $F$ to a \hc diagram of shape $\hat\calc$. It induces a map $H_*\colon\cyl(F)\to\cyl(G)$ which is well-defined in $\mathbf{Top}$ and functorial for natural transformations.

The homotopy class of the map $H_*$ defined in this way agrees with the class induced by $H$ when considered as a \hc transformation. This follows from the Remark above since one can define a \hc diagram on $\widehat{P\times [1]}$.
\end{remark}

To prove Lemma \ref{lemma3}, one could directly work with the universal property: The fact that the homotopy coherent nerve of $\mathbf{Top}$ is weakly Kan can be used to show that a \hc diagram $F$ over $\calc\times [i]$ and an extension of $F\vert_{\calc_{i\dots n}}$ over the cone determine an extension of the whole diagram $F$ over the cone, which is unique in a suitable sense. 

However, the relation with the Bousfield-Kan homotopy colimit provides a shortcut. In fact, Lemma \ref{lemma3} is direct consequence of the following result, which generalizes the classical cofinality theorem by Bousfield-Kan \cite[Theorem XI.9.2]{Bousfield-Kan(1972)}. I am grateful to Bill Dwyer for pointing this out to me.

\begin{lemma}\label{lem:cofinality}
Let $f\colon \cald\to\calc$ be functor between small categories. The following are equivalent:
\begin{enumerate}
\item For any \hc diagram $F$ on $\calc$, the canonical map
\[\Hocolim (f^*F)\to \Hocolim(F)\]
is a homotopy equivalence.
\item For all $c\in \calc$, the geometric realization of the comma category $c/f$ is contractible.
\end{enumerate}
\end{lemma}

\begin{proof}
First notice that condition (ii) can be replaced by the condition (ii') that for all $c\in\calc$, we have
\[\Hocolim_\cald \vert\calf_\bullet\calc(c, f(-))\vert \simeq *.\]
In fact, the augmentation defines an $\calf_\bullet\calc$-natural homotopy equivalence
\[\varepsilon\colon \vert\calf_\bullet\calc(c, f(-))\vert\to \calc(c, f(-)),\]
and 
\[\Hocolim  \calc(c, f(-)) \simeq \hocolim  \calc(c, f(-)) = \vert c/f\vert.\]

To see that (i) implies (ii'), apply (i) to $F=\vert\calf_\bullet \calc(c,-)\vert$ and notice that
\[\Hocolim F \simeq \hocolim  \calc(c,-)=\vert c/\calc\vert\simeq *.\]

The implication from (ii') to (i) is shown in the same way as the classical cofinality theorem.
\end{proof}

We now turn to rectifications. Let $F$ be a \hc diagram of the shape $\calc$. 

\begin{definition}
A rectification of $F$ is a \hc transformation $H\colon F\Rightarrow G$  such that 
\begin{enumerate}
\item $G$  is a strict diagram, and
\item  $H$ is invertible in $\Coh(\calc, \mathbf{Top})$.
\end{enumerate}
\end{definition}

Rectifications always exist, by the following Theorem which is goes back to Vogt \cite{Vogt(1973)}:
\begin{theorem}[{\cite{Cordier-Porter(1986)}}]
\begin{enumerate}
\item A \hc transformation $H\colon F\Rightarrow G$ is invertible in $\Coh(\calc,\mathbf{Top})$ if and only if it is a level homotopy equivalence, i.e.~each map $H_c\colon F(c)\to G(c)$ is  a homotopy equivalence.
\item Denote by $\Ho(\mathbf{Top}^\calc)$ the category obtained from the functor category $\mathbf{Top}^\calc$ by inverting the level homotopy equivalences. The induced functor
\[\Ho(\mathbf{Top}^\calc)\to\Coh(\calc,\mathbf{Top}).\]
is an equivalence of categories. 
\end{enumerate}
\end{theorem}

We are going to show that homotopy colimits allow to construct rectifications. Moreover, the homotopy colimit of a given \hc diagram is just the (ordinary) colimit of this ``standard rectification''.

Denote, for any object $c$ of $\calc$, by
\[\varphi_c\colon \calc/c\to \calc\]
the forgetful functor which sends the morphism $f\colon d\to c$ to $d$. For each $f\colon d\to c$ in $\calc$, the triangle
\[\xymatrix{
{\calc/d} \ar[rr]^{f_*} \ar[rd]_{\varphi_d} && {\calc/c} \ar[ld]^{\varphi_c} \\
& \calc
}\]
commutes. Thus any morphism $f\colon d\to c$ induces a map
\[\bar f\colon \cyl(\varphi_c^* F)\to \cyl(\varphi_d^* F).\]
It follows from naturality that $\overline{g\circ f}=\bar g \circ \bar f$ whenever $g$ and $f$ are composable. Hence we obtain a strict diagram $rF$ of shape $\calc$ with $rF(c)=\cyl(\varphi_c^* F)$.

\begin{proposition}
There is a ``standard'' rectification $H\colon F\Rightarrow rF$.
\end{proposition}

\begin{proof}
We need to construct a \hc diagram $H$ of shape $\calc\times [1]$ such that $H\vert_0=F$ while $H\vert_1 =rF$. Hence we need to construct natural maps
\[\calf_\bullet\bigl(\calc\times [1]\bigr)\bigl((d,0),(c,1)\bigr)\to\map\bigl(F(d), rF(c)\bigr).\]

Notice that 
\[\calf_\bullet\bigl(\calc\times [1]\bigr)\bigl((d,0), (c,1)\bigr) = \coprod_{f\colon d\to c} \calf_\bullet\bigl(\calc/c\times [1]\bigr)\bigl((f,0), (\id_c, 1)\bigr).\]
Now there is a functor
\[\calc/c\times [1]\to \widehat{\calc/c}\]
which is the identity on $\calc/c\times [0]$ and sends $(f,1)$ to the terminal object $*$. We may use this functor to pull back the \hc diagram on $\widehat{\calc/c}$ given by the mapping cylinder of $\varphi_c^*F$. Evaluation of this diagram gives a map
\[
\calf_\bullet\bigl(\calc/c\times [1]\bigr)\bigl((f,0), (\id_c, 1)\bigr)
 \to \map\bigl(F(d), \cyl(\varphi_c^* F)\bigr)=\map\bigl(F(d), rF(c)\bigr)
\]

It is obvious that these assignments are natural in $c$. To show naturality in $d$, we have to show that for $h\in\calf_n\calc(c,c')$ the following diagram is commutative:
\[\xymatrix@C=2.5ex{
{\calf_n(\calc\times[1])((d,0),(c,1))} \ar[r] \ar[d]^{h_*}  
 & {\underset{f\colon d\to c}{\coprod} \calf_n\widehat{\calc/c}(f,*)} \ar[d]^{h_*} \ar[r]
 & {\map(F(d)\times\Delta^n, rF(c))} \ar[d]^{h_*}
\\
{\calf_n(\calc\times[1])((d,0),(c',1))} \ar[r]   
 & {\underset{g\colon d\to c}{\coprod} \calf_n\widehat{\calc/c'}(g,*)}  \ar[r]
 & {\map(F(d)\times\Delta^n, rF(c'))} 
}\]
Here the right-hand vertical map $h_*\colon rF(c)=\cyl (\varphi_c^*F)\to \cyl(\varphi_{c'}^F)=rF(c')$ is given by the naturality of the homotopy colimit with respect to the index category. By definition of this naturality, the right-hand square commutes.
The commutativity of the left-hand square is obvious.

Finally let $\alpha\colon (c,0)\to(c,1)$ be the unique morphism in $\calc\times [1]$ which is the identity on $c$. The image of $\alpha$ under $H$ is the canonical map
\[F(c)\to rF(c)=\cyl(\varphi_c^* F)\]
which includes the value of $\varphi_c^* F$ at the terminal object $\id_c$ of $\calc/c$ into the homotopy colimit. We proved in Lemma \ref{lemma1} that this map is a homotopy equivalence. 
\end{proof}

The following two results are included for completeness; they are not needed elsewhere in the text.

\begin{proposition}\label{mapping_cylinder_as_a_colimit}
The collection of canonical mappings $\bar\varphi_c\colon rF(c)=\cyl(\varphi_c^* F)\linebreak[1]\to\cyl(F)$ (for $c\in\calc$) induce a homeomorphism
\[\colim (rF)\xrightarrow{\cong} \cyl(F).\]
\end{proposition}

\begin{proof}
We have to show that an extension of $F$ to a \hc diagram $G$ of shape $\hat\calc$ with $G(*)=X$ is uniquely determined by a compatible collection of maps 
\[\alpha_c\colon rF(c)=\cyl(\varphi_c^* F)\to X.\] 
Each of the maps $\alpha_c$ determines an extension of $\varphi_c^* F$ to a \hc diagram $G_c$ of shape $\widehat{\calc/c}$; the compatibility assumption translates to the fact for all strings $d\xrightarrow{f} c\xrightarrow{g} c'$ in $\calc$, the following triangles commute:
\[\xymatrix{
{\vert\calf_\bullet\widehat{\calc/c}(f,*) \vert} \ar[rr]^{g_*} \ar[rd]^{G_c} 
 && {\vert\calf_\bullet\widehat{\calc/c'}(g\circ f,*) \vert} \ar[ld]_{G_{c'}}\\
 & {\map(F(d),X)}
}\]
This corresponds, for any $d$, to giving a map
\begin{equation}\label{eq:induced_map}\colim_c \left(\coprod_{f\in\calc(d,c)} \left\vert \calf_\bullet\widehat{\calc/c}(f,*)\right\vert\right) \to \map(F(d),X).
\end{equation}

Now notice that the functors $\varphi_c$ induce for all $n$ a bijection
\[\colim_c \left(\coprod_{f\in\calc(d,c)} \calf_n\widehat{\calc/c}(f,*)\right) \to \calf_n\hat\calc(d,*)\] 
Passing to the geometric realization, we see that the left-hand side of \eqref{eq:induced_map} is indeed homeomorphic to $\vert\calf_\bullet\hat\calc(d,*)\vert$. Hence the maps $\eqref{eq:induced_map}$, for $d\in \calc$, precisely determine an extension of $F$ to a \hc diagram $G$ of shape $\hat\calc$, such that $G(*)=X$.
\end{proof}

As an application we can show that the standard rectification satisfies a certain ``cofibrancy condition''. Denote, for any object $c$ of $\calc$, by $\partial c$ the full subcategory of $\calc/c$ of all objects different from the terminal object $\id_c$.

\begin{corollary}\label{cor:standard_rectification_is_cofibrant}
Given $F$ as above, the rectification $rF$ constructed in this section satisfies the following condition:
For any object $c$ of $\calc$, the canonical map
\[\colim_{(f\colon d\to c)\in\partial c}rF(d) \to rF(c)\]
is a cofibration.
\end{corollary}

\begin{proof}
Let $f\colon d\to c$ be given. Recall that $rF(d)=\cyl(\varphi_d^* F)$
where $\varphi_d^* F$ is the \hc diagram pulled back to $\calc/d$. Now note that there is a commutative diagram of functors
\[\xymatrix{
(\calc/c)/f \ar[rr]^{\varphi_f} \ar[d]^\cong && {\calc/c} \ar[d]^{\varphi_c} \\
{\calc/d} \ar[rr]^{\varphi_d} && {\calc}
}\]
Thus we may identify $\varphi_d^* F$ with $\varphi_f^*\varphi_c^* F$, and we obtain
\[rF(d)=\cyl(\varphi_d^* F) \cong \cyl(\varphi_f^*\varphi_c^* F).\]
It now follows from Proposition \ref{mapping_cylinder_as_a_colimit} that
\[\colim_{(f\colon d\to c)\in\partial c} rF(d) = \colim_{(f\colon d\to c)\in\partial c} \cyl(\varphi_f^*\varphi_c^* F) \cong \cyl((\varphi_c^* F)\vert_{\partial c}).\]
On the other hand it follows from Lemma \ref{lemma1} that the inclusion
\[\cyl((\varphi_c^* F)\vert_{\partial c})\to \cyl(\varphi_c^* F)=rF(c)\]
is a cofibration.
\end{proof}

\section{The homotopy colimit over finite posets}\label{section:geometric_properties}

Throughout this section, let $P$ be a finite poset. $P$ will be considered as a small category with object set $P$, where there is a unique morphism from $x$ to $y$ if and only if $x\leq y$.

\begin{proposition}\label{prop:cylinder_is_compact_ENR}
Let $F$ be a \hc diagram of shape $P$, such that $F(x)$ is a compact ENR for all $x\in P$. Then $\cyl(F)$ is also a compact ENR.
\end{proposition}

This proposition will be proved by induction over the dimension of $P$, which by definition is the largest number $n$ such that there is a strictly increasing sequence
\[x_0\lneqq x_1\lneqq\dots\lneqq x_n.\]
(Thus the dimension of $P$ is just the dimension of the geometric realization of $P$.) To carry out the induction, we need the following result, in which $M\subset P$ denotes the subset of maximal elements, and, for $x\in P$, the symbol $\bar x$ denotes the subset of $P$ of all elements less than or equal to $x$. 

\begin{lemma}\label{lemma2}
Let $F$ be a \hc diagram of shape $P$. An extension of $F$ over the cone $\hat P$ determines and is determined by extensions of $F\vert_{P-M}$ and $F\vert_{\bar m-m}$, for $m\in M$, over the respective cones that agree over their common intersection. In other words, the inclusion-induced diagram 
\[\xymatrix{
{\coprod_{m\in M} \cyl(F\vert_{\bar m- \{m\}})} \ar[rr] \ar@{_{(}->}[d] 
  && {\cyl (F\vert_{P-M})} \ar@{_{(}->}[d]\\
{\coprod_{m\in M} \cyl(F\vert_{\bar m})} \ar[rr]
 && {\cyl (F)}
}\]
is a push-out (where the vertical maps are cofibrations).
\end{lemma}

This Lemma is an immediate consequence of the following more general statement:

\begin{lemma}\label{lemma2_generalized}
Suppose that
\begin{equation}\label{eq:diagram_of_categories}
\xymatrix{
{\calc_0} \ar[rr]^{f_1} \ar[d]^{f_2} && {\calc_1} \ar[d]^{g_1}\\
{\calc_2} \ar[rr]^{g_2} && {\calc}
}
\end{equation}
is a commutative diagram of small categories which induces a push-out diagram on the level of the nerves. Suppose moreover that all the functors appearing in \eqref{eq:diagram_of_categories} send non-identity morphisms to non-identity morphisms. Then, for a \hc diagram $F\colon \calf_\bullet\calc\to\mathbf{Top}$, the induced square
\[\xymatrix{
{\cyl(g_0^* F)} \ar[rr] \ar[d]
 && {\cyl(g_1^* F)} \ar[d] \\
{\cyl(g_2^* F)} \ar[rr]
 && {\cyl(F)}
}\]
is a push-out, where $g_0=g_1\circ f_1=g_2\circ f_2$.
\end{lemma}

To prove Lemma \ref{lemma2_generalized}, we will make use of the following construction: Recall that an element in $\calf_n\calc(c,d)$ may be seen as a word of non-identity morphisms of $\calc$ together with $n+1$ levels of parentheses. It follows that for each pair of natural numbers $n,m$ there is a projection map
\[\alpha\colon N_m \calf_n\calc\to \coprod_{k\in\IN} N_k\calc\]
defined on the $m$-nerve by forgetting the parentheses. It is easy to verify the following statement:

\begin{lemma}
If $f\colon\calc\to\cald$ is a functor that sends non-identity morphisms to non-identity morphisms, then the square
\[\xymatrix{
{N_m\calf_n\calc} \ar[d]^\alpha \ar[rr]^f 
&& {N_m\calf_n\cald} \ar[d]^\alpha
\\
{\coprod_{k\in\IN} N_k\calc} \ar[rr]^f
&& {\coprod_{k\in\IN} N_k\cald}  
}\]
is a pull-back.
\end{lemma}

\begin{corollary}\label{cor:corollary_in_proof_of_lemma2} 
Under the assumptions of Lemma \ref{lemma2_generalized}, for each $m$ the following diagram is a push-out:
\[\xymatrix{
N_m \calf_\bullet \calc_0 \ar[rr]^{f_1} \ar[d]^{f_2} 
  && N_m\calf_\bullet\calc_1 \ar[d]^{g_1}\\
N_m \calf_\bullet \calc_2 \ar[rr]^{g_2} 
  && N_m\calf_\bullet\calc
}\]
\end{corollary}

\begin{proof}[Proof of Corollary]
It is enough to show that statement holds in each simplicial degree $n$. This is a consequence of the lemma together with the fact that if a set $B$ is a push-out of $B_1$ and $B_2$ along $B_0$, and $p\colon E\to B$ is a map of sets, then $E$ is a push-out of the corresponding pull-backs $E\vert_{B_1}$ and $E\vert_{B_2}$ along $E\vert_{B_0}$.
\end{proof}

\begin{proof}[Proof of Lemma \ref{lemma2_generalized}]
First notice that if \eqref{eq:diagram_of_categories} induces a push-out on the nerves, then so does the ``extended'' version
\[\xymatrix{
{\hat\calc_0} \ar[rr]^{f_1} \ar[d]^{f_2} && {\hat\calc_1} \ar[d]^{g_1}\\
{\hat\calc_2} \ar[rr]^{g_2} && {\hat \calc}  
}\]
Thus we may apply Corollary \ref{cor:corollary_in_proof_of_lemma2} to this diagram. Specializing to $m=1$ and looking only at the part of the 1-nerve of those maps that end at the terminal object, one sees that a map
\[a\colon \coprod_{c\in\calc} \vert \calf_\bullet\hat\calc(c,*)\vert \times F(c) \to Z\]
is uniquely determined by its restrictions
\[a_i\colon \coprod_{c_i\in\calc_i} \vert\calf_\bullet\hat\calc_i(c_i,*)\vert\times F(g_i(c_i)) \to Z.\]
The map $a$ determines an extension of $F$ to a homotopy coherent diagram if and only if $a$ is compatible with the $\calc$-operation, i.e.~factors over the coequalizer of
\[\coprod_{c,d\in \calc} \vert\calf_\bullet\hat \calc(d,*)\vert\times\vert\calf_\bullet\calc(c,d)\vert\times F(c)
\rightrightarrows\coprod_{c\in\calc} \vert\calf_\bullet\hat\calc(c,*)\vert\times F(c).\] 
This follows again from Corollary \ref{cor:corollary_in_proof_of_lemma2}, applied to $m=2$, since the first two factors in the left-hand side may again be expressed in terms of (a simplicial subset of) the 2-nerve of $\calf_\bullet \hat\calc$.
\end{proof}

\begin{proof}[Proof of Proposition \ref{prop:cylinder_is_compact_ENR}]
If $\dim(P)=0$, then $\cyl(F)$ is a finite coproduct of compact ENRs and thus a compact ENR itself (this follows from \cite[Theorem  E.3]{Bredon(1993)}). 

For the inductive step, let $\dim(P)=n$ and suppose that the statement holds for all posets of dimension less than $n$. Consider the push-out diagram as given by Lemma \ref{lemma2}:
\[\xymatrix{
{\coprod_{m\in M} \cyl(F\vert_{\bar m- \{m\}})} \ar[rr] \ar[d] 
  && {\cyl (F\vert_{P-M})} \ar[d]\\
{\coprod_{m\in M} \cyl(F\vert_{\bar m})} \ar[rr]
 && {\cyl (F)}
}\] 
By the inductive assumption, the space in the upper right-hand corner is a compact ENR. Lemma \ref{lemma1} shows that the space appearing in the lower left corner is the (ordinary) mapping cylinder of the map
\[\coprod_{m\in M}\cyl(F\vert_{\bar m-\{m\}}) \to \coprod_{m\in M} F(m)\]
between compact ENRs. Thus it is a compact ENR by \cite[Corollary E.7]{Bredon(1993)}. Moreover the inclusion of $\coprod\cyl(F\vert_{\bar m-\{m\}})$ into $\coprod \cyl(F\vert_{\bar m})$ is a cofibration. Hence, by \cite[chapter VI, \S 1]{Hu(1965)} the adjunction space $\cyl(F)$ is also a compact ENR. 
\end{proof}

\begin{corollary}\label{cor:rectification_with_ENR_property}
Let $F$ be a \hc diagram of shape $P$, such that for all $x\in P$, the space $F(x)$ is homotopy equivalent to a compact ENR. Then $F$ has a rectification $G$ such that for all $x\in P$, $G(x)$ is a compact ENR.
\end{corollary}

\begin{proof}
For each $x\in P$, choose a homotopy equivalence $\varphi_x\colon F(x)\to F'(x)$ to some compact ENR. Then there is a \hc diagram $F'$ with prescribed spaces $F'(x)$ and a \hc transformation $F\to F'$ with prescribed maps $\varphi_x$ \cite[\S 2]{Cordier-Porter(1988)}. Let $G:=rF'$ be the standard rectification of $F'$ as in section \ref{section:hc_transformations}, i.e.~
\[G(x)=\cyl(F'\vert_{\bar x}).\]
For each $x$, the space $G(x)$ is a compact ENR by Proposition \ref{prop:cylinder_is_compact_ENR}. 
\end{proof}

\begin{proposition}\label{prop:AF_on_generalized_mapping_cylinders}
Let $F, G$ be strict diagrams of the shape $P$ consisting entirely of compact ENRs and homotopy equivalences, and let $f\colon F\to G$ be a natural transformation. Suppose that for all $x\in P$, the map $f_x\colon F(x)\to G(x)$ is an approximate fibration. Then $f_*\colon\cyl(F)\to\cyl(G)$ is also an approximate fibration.
\end{proposition}

Let us first deal with the special case that $P=\{0<1\}$. A slightly more general statement is:

\begin{lemma}\label{lem:AF_on_ordinary_mapping_cylinders}
Let
\[\xymatrix{
 A \ar[rr]^f \ar[d]^p && X \ar[d]^q\\
 B \ar[rr]^g && Y
}\]
be a weak homotopy pull-back square, with all spaces ENRs, and $A, B$ compact. If $p$ and $q$ are approximate fibrations, then so is the induced map $r\colon\Cyl(f)\to\Cyl(g)$ between the mapping cylinders.
\end{lemma}

Recall that a commutative square of spaces is said to be a weak homotopy pull-back if the induced map on vertical (or horizontal) homotopy fibers is a weak homotopy equivalence.

\begin{remark}
The statement of Lemma \ref{lem:AF_on_ordinary_mapping_cylinders} becomes wrong if we drop the ENR condition on $A$ and $X$. To see this, let $W$ be a space which is weakly contractible but not contractible (such as the Warsaw circle). Then the diagram
\[\xymatrix{
 W \ar[rr]^f \ar[d]^p && {*} \ar[d]^q\\
 {*} \ar[rr]^g && {*}
}\]
is a weak homotopy pull-back. Suppose that the induced projection map $r$ from the cone $CW$ to $I=[0,1]$ was an approximate fibration. There is a contraction $W\times I\to CW$ which collapses $W\subset CW$ onto a point $w\in W$, passing through the cone point; if $r$ was an approximate fibration, we could use the approximate lifting property to obtain a contraction of $W$ in $W\times [0,\varepsilon)\subset CW$. This contradicts the fact that $W$ is not contractible.
\end{remark}

\begin{proof}[Proof of Lemma \ref{lem:AF_on_ordinary_mapping_cylinders}]
By \cite[Corollary E.7]{Bredon(1993)}, the mapping cylinders of $f$ and $g$ are ENRs. Using \cite[Theorem 12.15]{Hughes-Taylor-Williams(1990)}, it is enough to show that for all elements $U$ a given basis of the topology of $\Cyl(g)$, the diagram
\[\xymatrix{
{\Cyl(f)\vert_U} \ar[d]^{r\vert_U} \ar[rr]&& {\Cyl(f)} \ar[d]^r\\
U \ar[rr]&& {\Cyl(g)}
}\]
(in which the horizontal maps are the inclusions) is a weak homotopy pull-back, i.e.~induces a weak homotopy equivalence on the vertical homotopy fibers.

Consider the basis of $\Cyl(g)$ which consists of all subsets of either of the following forms:
\begin{enumerate}
\item $V\times [0,a)$ for $0<a<1$ and $V\subset B$ open,
\item $V\times (a,b)$ for $0<a<b<1$ and $V\subset B$ open,
\item $g\inv(W)\times (a,1)\cup W$ for $0<a<1$ and $W\subset Y$ open.
\end{enumerate}
In case (i), we have $\Cyl(f)\vert_U=p\inv(V)\times [0,a)$, so in the diagram
\begin{equation}\label{eq:diagram_of_cyls}
\xymatrix{
p\inv(V) \ar[rr]^\simeq \ar[d]^{p\vert_V} && {\Cyl(f)\vert_U} \ar[d]^{r\vert_U} \ar[rr]&& {\Cyl(f)} \ar[d]^r \ar[rr]^\simeq && X \ar[d]^q\\
V \ar[rr]^\simeq && U \ar[rr]&& {\Cyl(g)}\ar[rr]^\simeq && Y
}
\end{equation}
the left-hand square is a weak homotopy pull-back. The same is true for the right-hand square. Finally, in the diagram
\[\xymatrix{
p\inv(V) \ar[rr] \ar[d]^{p\vert_V} && A \ar[rr]^f \ar[d]^p && X \ar[d]^q\\
V \ar[rr] && B \ar[rr]^g && Y
}\]
the left-hand square is a weak homotopy pull-back, as $p$ is an approximate fibration, and the right-hand square is a weak homotopy pull-back by assumption. So the total square, which agrees with the total square of \eqref{eq:diagram_of_cyls}, is a weak homotopy pull-back. Hence the square in the middle of \eqref{eq:diagram_of_cyls} is a weak homotopy pull-back, as required.

The argument in case (ii) is identical. The case (iii) is a similar argument, using this time that $q$ is an approximate fibration.
\end{proof}

To take profit, in the general case, of the statement we just proved, we will make use of the following result:

\begin{fact}\label{fact:quasi-fibration}
Quillen has shown that if $F$ is a strict diagram (over an arbitrary index category $\calc$) consisting entirely of homotopy equivalences, then the canonical projection 
\[\hocolim F\to \hocolim *=\vert \calc\vert\]
is a quasi-fibration. This implies that if $F\to G$ is natural transformation between such diagrams, then for all objects $c$ the obvious square
\[\xymatrix{
 F(c) \ar[rr]\ar[d] && {\hocolim(F)} \ar[d]\\
 G(c) \ar[rr] && {\hocolim(G)}
}\]
is a weak homotopy pull-back. By homotopy invariance, the same statement thus holds with hocolim replaced by the Hocolim.
\end{fact}

\begin{proof}[Proof of Proposition \ref{prop:AF_on_generalized_mapping_cylinders}]
Again the proof uses Lemmas \ref{lemma1} and \ref{lemma2} for an induction on the dimension of $P$. If $\dim(P)=0$, then $\cyl(F)$ and $\cyl(G)$ are coproducts and $f_*$ is the induced map, hence an approximate fibration.

Suppose inductively that $\dim(P)=n>0$ and that the statement holds for all posets of dimension less than $n$. Denote by $M\subset P$ the set of maximal elements, and for $m\in M$, let $\bar m\subset P$ be the subset of elements which are less than or equal to $m$. By our assumptions, the maps
\[\cyl(F\vert_{\bar m-\{m\}})\to \cyl(G\vert_{\bar m-\{m\}}),\; \cyl(F\vert_{P-M})\to\cyl(G\vert_{P-M})\]
are approximate fibrations. We claim that the map
\[\cyl(F\vert_{\bar m})\to\cyl(G\vert_{\bar m})\]
is an approximate fibration, too. In fact, using Lemma \ref{lemma1}, this map is induced from
\[\xymatrix{
{\cyl(F\vert_{\bar m-\{m\}})} \ar[d] \ar[rr] && F(m) \ar[d]\\
{\cyl(G\vert_{\bar m-\{m\}})} \ar[rr] && G(m)
}\]
on the level of horizontal mapping cylinders. It follows from Lemma \ref{lem:AF_on_ordinary_mapping_cylinders} and Fact \ref{fact:quasi-fibration} that it is an approximate fibration.

By \cite[Theorem 12.15]{Hughes-Taylor-Williams(1990)}, it is enough to show that 
\[f_*\colon\cyl(F)\to\cyl(G)\]
is locally an approximate fibration. Suppose for simplicity that $M$ consists of a single element. Lemma \ref{lemma2} then shows that $\cyl(G)$ is a union of $\cyl(G\vert_{\bar m})$ and $\cyl(G\vert_{P-M})$. We already know that $f_*$ restricted to $\cyl(G\vert_{\bar m})$, and hence restricted to any open subset of it, is an approximate fibration. It is therefore enough to show that the restriction of $f_*$ over a neighborhood of $\cyl(G\vert_{P-M})$ is an approximate fibration as well.

Such a neighborhood is given by the adjunction space
\[\cyl(G\vert_{P-M})\cup_{\cyl(G\vert_{\bar m-\{m\}})\times 0} \cyl(G\vert_{\bar m-\{m\}})\times [0,\varepsilon]\]
for $0<\varepsilon<1$. Notice that this adjunction space can be identified with a mapping cylinder. In fact, more generally, the restriction of $f_*$ over this adjunction space may be identified with the map induced from
\[\xymatrix{
{\cyl(F\vert_{\bar m-\{m\}})} \ar[d] \ar[rr] 
  && {\cyl(F\vert_{P-M})} \ar[d]\\
{\cyl(G\vert_{\bar m-\{m\}})} \ar[rr] 
  && {\cyl(G\vert_{P-M})} 
}\]
on the level of horizontal mapping cylinders. Therefore it is an approximate fibration, again by Lemma \ref{lem:AF_on_ordinary_mapping_cylinders} and Fact \ref{fact:quasi-fibration}. 

If $M$ consists of more than one element, Lemma \ref{lemma2} shows that $\cyl(G)$ can be obtained by successively adjoining the spaces $\cyl(G\vert_{\bar m-\{m\}})$ for $m\in M$. In this case the above argument applies for each of the steps.
\end{proof}

\section{The factorization result}\label{section:main_result}

The goal of this section is to prove the following theorem:

\begin{theorem}\label{main_theorem}
Let $p\colon E\to B$ be a fibration over a finite simplicial complex. Suppose that the fibers of $p$ are homotopy equivalent to finite simplicial complexes. Then there exists a compact ENR $X$ and a commutative diagram
\[\xymatrix{
X \ar[rr]^\varphi \ar[rd]_q && E \ar[ld]^p\\
 & B
}\]
where $q$ is an approximate fibration and $\varphi$ is a homotopy equivalence.
\end{theorem}

Denote by $P$ the (finite) poset of simplices of the simplicial complex $B$. For $\sigma\in P$, let $E_\sigma\subset E$ be the restriction of $E$ over $\sigma$. The rule
\[\sigma\mapsto E_\sigma\]
defines a strict diagram on $P$. By our assumptions, $E_\sigma$ is homotopy equivalent to a finite simplicial complex. So Corollary \ref{cor:rectification_with_ENR_property} guarantees the existence of a rectification $H \colon E_\sigma\Rightarrow X_\sigma$ such that $X_\sigma$ is a compact ENR for each $\sigma$.

We are going to construct a homotopy commutative diagram
\[\xymatrix{
E \ar[d]^p && {\cyl_\sigma(E_\sigma)} \ar[ll]_{h_E}^\simeq \ar[r]^{H_*}_\simeq \ar[d] & {\cyl_\sigma(X_\sigma)} \ar[ddl]^{\pi_*} \\
B \ar@{=}[d] && {\cyl_\sigma(\sigma)} \ar[d]_{\pi'_*} \ar[ll]_{h_B}^\simeq\\
B && {\cyl_\sigma(*)} \ar[ll]_\iota^\cong
}\]
in which $\iota$ is a homeomorphism, and the maps $\pi_*$ and $\pi'_*$ are induced from the unique natural transformation to the constant diagram with value $*$.

The theorem now follows from this diagram. In fact if we let $X:=\cyl_\sigma(X_\sigma)$, then $X$ is a compact ENR by Proposition \ref{prop:cylinder_is_compact_ENR}. Inverting $H_*$ yields a map $\varphi\colon X\to E$ which we may change by the homotopy lifting property of $p$ so that $p\circ \varphi=\iota\circ\pi_*=:q$. The maps in the right triangle are given by the naturality of the 
homotopy colimit; in particular it follows from Proposition \ref{prop:AF_on_generalized_mapping_cylinders} that the map $\pi_*$ (and hence also $q$) is an approximate fibration. 

The two maps $h_E$ and $h_B$ are induced by the obvious (even strict) extensions of the corresponding diagrams. The commutativity of the upper left-hand square follows easily from Remark \ref{rem:commutativity_of_triangle}.

Thus we are left to do the following:
\begin{enumerate}
\item Show that $h_E$ and $h_B$ are homotopy equivalences.
\item  Construct the homeomorphism $\iota$ and to prove that the lower left-hand square commutes up to homotopy.
\end{enumerate}

We start with (i). We only need to show that $h_E$ is a homotopy equivalence. This is certainly true if $B$ is 0-dimensional.

Suppose that the result holds for base spaces of the dimension $n$, and that $\dim(B)=n+1$. Then the maximal elements of $P$ correspond precisely to the $(n+1)$-simplices of $B$. By  Lemma \ref{lemma2}, we have therefore a push-out diagram
\[\xymatrix{
{\coprod_{\tau\in M} \cyl_{\sigma\in \bar \tau-\{\tau\}}(E_\sigma)} \ar[rr]\ar@{_{(}->}[d]
 && {\cyl_{\sigma\in P-M}(E_\sigma)} \ar@{_{(}->}[d]\\
{\coprod_{\tau\in M} \cyl_{\sigma\in \bar \tau}(E_\sigma)} \ar[rr]
 && {\cyl_{\sigma\in P}(E_\sigma)}
}\]
Similarly $E$ decomposes as the following push-out
\[\xymatrix{
{\coprod_{\tau\in M} E\vert_{\partial \tau}} \ar[rr] \ar@{_{(}->}[d]
 && E\vert_{B_n} \ar@{_{(}->}[d] \\
{\coprod_{\tau\in M} E_{\tau}} \ar[rr]
 && E
}\]
where $B_n\subset B$ denotes the $n$-skeleton and $\partial\tau\subset\tau$ is the $n$-skeleton of $\tau$. Also $h_E$ decomposes.

The restrictions of $h_E$ to the spaces on the top are homotopy equivalences by the inductive assumption. On the other hand it follows from Lemma \ref{lemma1} that the restriction of $h_E$ to the the lower left-hand corner is also a homotopy equivalence. By the glueing lemma, the map $h_E$ between the lower right-hand corners is therefore a homotopy equivalence, too. This proves (i).

Now we come to (ii). To obtain the map $\iota$, we need to construct an extension of the constant diagram $*$ over $P$ to a \hc diagram $G$ over $\hat P$ such that $G(*)=B$. Intuitively the extension maps the point $*$ belonging to $\sigma$ to the barycenter of $\sigma$ in $B$. If $\sigma<\tau$, there is a canonical (linear) path between the barycenters of $\sigma$ and $\tau$; this rule also applies to higher coherences.

Formally this looks as follows. An extension of the constant diagram $*$ of shape $P$ over the cone is given by a natural transformation $\vert\phi_\bullet(\sigma)\vert\to\map(*,B)=B$. (See section \ref{section:hocolim} for notation.)
Recall the natural transformation $\tau_\bullet\colon \phi_\bullet(\sigma)\to N_\bullet(\sigma\backslash P)$ to the comma category from section \ref{section:relation_to_BK}. The transformation we are looking for is the following composite:
\[g_\sigma\colon \vert\phi_\bullet(\sigma)\vert\xrightarrow{\vert\tau_\bullet\vert} \vert \sigma\backslash P\vert \xrightarrow{\vert\varphi_\sigma\vert} \vert P\vert\cong B,\]
where the latter homeomorphism is the one between $B$ and its simplicial subdivision, and the functor $\varphi_\sigma$ is the forgetful one from section \ref{section:hc_transformations}.

An inductive argument similar to the one we used in (i) now shows that $\iota$ is indeed a homeomorphism.

We still need to show that the lower left-hand square commutes. To do this, we are going to construct a \hc diagram $G$ of shape $\widehat{P\times [1]}$ which restricts to the one represented by $\iota$ over $\hat P_0$ (let us call it $F_0$) and to the one (call it $F_1$) represented by $h_B$ over $\hat P_1$. By Remark \ref{rem:commutativity_of_triangle}, this will show that $h_B\circ s_*\simeq \iota$, where $s=G\vert_{P\times [1]}\colon *\Rightarrow\sigma$. But the constant diagram $*$ is terminal in $\Coh(P,\mathbf{Top})$, so the diagram $\sigma$ is also terminal, and $s$ is an inverse of $\pi'$. This implies that $h_B\simeq \iota\circ\pi'_*$.

The intuitive idea to define $G$ is again to map the point $*$ belonging to $\sigma$ to the barycenter of $\sigma$, now considered as a point in $\sigma$ rather than $B$. To carry out the formal construction, notice first that the inclusion $\hat P=\hat P_0\subset \widehat{P\times [1]}$ is split by 
\[r\colon  \widehat{P\times [1]} \to \hat P,\quad (x,0)\mapsto x, \quad (x,1)\mapsto *.\]

Define
\[G':=r^* F_0,\]
thus the pull-back of $F_0$ along the functor $r$. Since $r$ splits the inclusion, $G'$ is a \hc diagram on $\widehat{P\times [1]}$ that extends $F_0$. However, $G'(\sigma, 1)=B$ while $F_1(\sigma)=\sigma$ so $G'$ does not extend $F_1$. But observe that the maps
\[\vert\calf_\bullet(\widehat{P\times [1]})((\sigma,0),(\mu,1))\vert\to \vert\map(G'(\sigma,0),G'(\mu,1))\vert=\map(*,B)=B\]
given by the functor $G'$ actually take values in $\mu\subset B$. (This follows from the fact that the diagram
\[\xymatrix{
{\vert \calf_\bullet (\widehat{P\times [1]})((\sigma,0),(\mu,1))\vert} \ar[r]^(.65)r  
& {\vert \calf_\bullet\hat P(\sigma,*) \vert} \ar[r]^(.55){\vert\tau_\bullet\vert}
& {\vert \sigma\backslash P \vert} \ar[r]^{\vert\varphi_\sigma\vert}
& {\vert P\vert} \ar[r]^\cong
& B
\\
{\vert \calf_\bullet (\widehat{\bar\mu\times [1]})((\sigma,0),(\mu,1))\vert} \ar[r]^(.65)r  \ar[u]
& {\vert \calf_\bullet\hat {\bar\mu}(\sigma,*) \vert} \ar[r]^(.55){\vert\tau_\bullet\vert}\ar[u]
& {\vert \sigma\backslash \bar \mu \vert} \ar[r]^{\vert\varphi_\sigma\vert}\ar[u]
& {\vert \bar \mu\vert} \ar[r]^\cong\ar[u]
& \mu\ar[u]
}\]
with inclusions as vertical arrows is commutative.)

Hence if we let $G(\sigma, 1)=\sigma$, $G(\sigma,0)=*$, $G(*)=B$, and use the same operations for the morphisms as for $G'$, we obtain a \hc diagram over $P\times [1]$ as required.

\section{Simple homotopy type}\label{section:simple_homotopy_type}

Given a space $E$, a \emph{simple structure} $\xi$ on $E$ is the choice of a homotopy equivalence $f\colon E\to X$ to a compact ENR, where we identify $f$ with $f'\colon E\to X'$ if the Whitehead torsion $\tau(f'\circ f\inv)$ is zero. This definition is made in such a way that the Whitehead torsion of a homotopy equivalence
\[g\colon (E,\xi)\to (E',\xi')\]
between spaces with simple structures can be defined to be $\tau(f'\circ g\circ f\inv)$, where $f'$ and $f$ represent the simple structures $\xi'$ and $\xi$ respectively.

Let $p\colon E\to B$ be a fibration over a finite connected CW complex whose fiber is homotopy equivalent to a finite CW complex. In this situation, Farrell, L\"uck and the author \cite{Farrell-Lueck-Steimle(2010)} constructed a preferred simple structure on $E$. This construction depends on the choice of a point $b\in B$, a simple structure $\zeta$ on the fiber $F_b$, and a choice of \emph{spider} $s$ based at $b$, i.e.~a collection of paths
\[\gamma_c\colon b\rightsquigarrow x\]
indexed over the set of cells $c$ of $B$, from $b$ to some point $x$ in the interior of the cell $c$. The resulting construction is denoted $\xi(b,\zeta, s)$. In favorable cases, e.g.~if $p$ is 2-connected, then $\xi(b,\zeta,s)$ is independent of the choice of $(b,\zeta,s)$.

Now Theorem \ref{main_theorem} provides the total space $E$ with a homotopy equivalence to a compact ENR. In this section we prove that a suitable version of this construction represents the simple homotopy type $\xi(b,\zeta, s)$.

To formulate the result, let $p\colon E\to B$ be a fibration over a compact connected simplicial complex, $b\in B$, let $\zeta$ be a simple structure on $F_b$ and let $s$ be a spider for $B$ based at $b$. Choose a representative $f\colon F_b\to Z$ for $\zeta$; for each simplex $\sigma$ of $B$, choose a fiber trivialization $t_\sigma\colon E_\sigma\to F_b\times \vert\sigma\vert$ along the path $\gamma_\sigma$. 

Denote by $P$ the poset of simplices of $B$ and let $F$ be the functor on $P$ 
\[\sigma\mapsto E_\sigma:= p\inv(\sigma).\]
There is \cite[\S 2]{Cordier-Porter(1988)} a \hc transformation $H$ from $F$ to a \hc diagram $G$ on $P$ such that $G(\sigma)=Z$ for all $\sigma\in P$, and 
\[H_\sigma\colon E_\sigma\xrightarrow{t_\sigma} F_b\times\vert\sigma\vert\xrightarrow{\Proj} F_b\xrightarrow{f} Z.\]

\begin{proposition}
The homotopy equivalence
\[E\xleftarrow{h_E} \cyl(F) \xrightarrow{H_*}  \cyl(G) \to \cyl(rG)\]
represents the simple structure $\xi(b,\zeta,s)$. Here the last map is induced by the standard rectification.
\end{proposition}

\begin{proof}
The proof goes by induction over the dimension of $B$. If $\dim(B)=0$, then the claim is evidently true. 

Suppose that the claim holds for base spaces up to dimension $n-1$, and that $\dim(B)=n>0$. Denote by $M$ the set of maximal elements of $P$. By Lemma \ref{lemma2}, the following square is a push-out square:
\[\xymatrix{
{\coprod_{m\in M} \cyl(rG\vert_{\bar m- \{m\}})} \ar[rr] \ar@{_{(}->}[d] 
  && {\cyl (rG\vert_{P-M})} \ar@{_{(}->}[d]\\
{\coprod_{m\in M} \cyl(rG\vert_{\bar m})} \ar[rr]
 && {\cyl (rG)}
}\]
(Notice that $r(G\vert_{\bar m})=(rG)\vert_{\bar m}$ and similarly for $\bar m-\{m\}$ and $P-M$.) Moreover the homotopy equivalence $E\simeq \cyl(rG)$ is induced on the level of horizontal push-outs by a diagram of the following shape:
\[\xymatrix@C=3ex{
{\coprod_{m\in M} \cyl(rG\vert_{\bar m})} 
  & {\coprod_{m\in M} \cyl(rG\vert_{\bar m- \{m\}})} \ar@{_{(}->}[l] \ar[r]
  &  {\cyl (rG\vert_{P-M})}
\\
{\coprod_{m\in M} \cyl(G\vert_{\bar m})} \ar[u] 
  & {\coprod_{m\in M} \cyl(G\vert_{\bar m- \{m\}})} \ar@{_{(}->}[l] \ar[r] \ar[u] 
  &  {\cyl (G\vert_{P-M})} \ar[u]
\\
{\coprod_{m\in M} \cyl(F\vert_{\bar m})} \ar[u] \ar[d]
  & {\coprod_{m\in M} \cyl(F\vert_{\bar m- \{m\}})} \ar@{_{(}->}[l] \ar[r] \ar[u] \ar[d]
  &  {\cyl (F\vert_{P-M})} \ar[u] \ar[d]
\\
{\coprod_{m\in M} E\vert_m} 
  & {\coprod_{m\in M} E\vert_{\partial m}} \ar@{_{(}->}[l] \ar[r]
  &  {E\vert_{\vert P-M\vert}}
}\]

By the inductive assumption, the composite homotopy equivalences in the right-hand and middle columns represent the simple structures $\xi(b,\zeta,s)$ associated to 
\[E\vert_{\vert P-M\vert}\quad\mathrm{and}\quad \coprod_{m\in M} E\vert_{\partial m}.\] By additivity \cite[Lemma 3.15]{Farrell-Lueck-Steimle(2010)}, it is therefore enough to show that the same is true for the left-hand column.

Let us analyze the left-hand column more closely. Recall that, by Lemma \ref{lemma1}, $\cyl(rG\vert_{\bar m})$ is a mapping cylinder whose bottom is $rG(m)=\cyl(G\vert_{\bar m})$. The upper arrow is the inclusion of the bottom into the whole cylinder and hence a simple homotopy equivalence. Similarly, $\cyl(G\vert_{\bar m})$ is simple homotopy equivalent to $G(m)=Z$ in a way that the following diagram is commutative.
\[\xymatrix{
{\cyl(F\vert_{\bar m})} \ar[d]^{\simeq_s} \ar[rr] && {\cyl(G\vert_{\bar m})} \ar[d]^{\simeq_s}\\
F(m) \ar[rr]\ar@{=}[d] &&  G(m)\ar@{=}[d]\\
E\vert _m \ar[rr]^{H_m} && Z
}\]
In other words, the simple structure on $E\vert_m$ given by the the left-hand colum is given by the homotopy equivalence $H_m$.

Let us now analyze the simple structure $\xi(b,\zeta,s)$ on $E\vert_m$. By fiber homotopy invariance \cite[Lemma 3.11]{Farrell-Lueck-Steimle(2010)} we may assume that $E\vert_m$ is in fact a product $F_b\times D^n$ and thus pulled back from a fibration over a point. By the compatibility with pull-back \cite[Lemma 3.16]{Farrell-Lueck-Steimle(2010)}, the simple structure $\xi(b,\zeta,s)$ on $E\vert_m$ is therefore given by the map $H_m$, too. Thus the two simple structures agree.
\end{proof}

\typeout{-----------------------  References ------------------------}

\end{document}